\documentclass[12pt]{article}
\usepackage{amsmath,amssymb,amsfonts,amsthm}
\newcounter{item}[section]
\newcounter{kirshr}
\newcounter{kirsha}
\newcounter{kirshb}
\newenvironment{enumroman}{\setcounter{kirshr}{1}
\begin{list}{(\roman{kirshr})}{\usecounter{kirshr}} }{\end{list}}
\newenvironment{enumarab}{\setcounter{kirshb}{1}
\begin{list}{(\arabic{kirshb})}{\usecounter{kirshb}} }{\end{list}}
\newtheorem{theorem}{Theorem}[section]

\newtheorem{lemma}[theorem]{Lemma}
\newtheorem{corollary}[theorem]{Corollary}
\newenvironment{demo}[1]{\noindent{\bf #1.}\upshape\mdseries}
{\nopagebreak{\hfill\rule{2mm}{2mm}\nopagebreak}\par\normalfont}
\theoremstyle{definition}

\newtheorem{example}[theorem]{Example}
\newtheorem{definition}[theorem]{Definition}
\def\R{\mathbb{R}}

\def\C{{\mathfrak{C}}}
\def\Fm{{\mathfrak{Fm}}}

\def\Nr{{\mathfrak{Nr}}}

\def\Sg{{\mathfrak{Sg}}}
\def\Fm{{\mathfrak{Fm}}}

\def\RCA{{\sf RCA}}
\def\RPEA{{\sf RPEA}}
\def\A{{\mathfrak{A}}}
\def\B{{\mathfrak{B}}}
\def\C{{\mathfrak{C}}}
\def\D{{\mathfrak{D}}}
\def\M{{\mathfrak{M}}}

\def\Sn{{\mathfrak{Sn}}}
\def\Bl{{\mathfrak{Bl}}}
\def\CA{{\sf CA}}

\def\SC{{\bf SC}}
\def\QEA{{\bf QEA}}

\def\Df{{\sf Df}}

\def\Dc{{\bf Dc}}

\def\K{{\bf K}}
\def\K{{\bf K}}

\def\RCA{{\sf RCA}}

\def\Rd{{\ Rd}}
\def\(R)RA{{\bf (R)RA}}
\def\RA{{\bf RA}}
\def\RRA{{\bf RRA}}
\def\Dc{{\bf Dc}}
\def\R{\mathbb{R}}
\def\F{{\sf F}}
\def\Dc{{\bf Dc}}

\def\Rl{{\sf Rl}}
\def\c #1{{\cal #1}}
 \def\CA{{\sf CA}}
\def\B{{\sf B}}
\def\G{{\sf G}}

\def\K{{\sf K}}
\def\tp{{\sf tp}}
 \def\Cm{{\mathfrak{Cm}}}
\def\Nr{{\mathfrak{Nr}}}

\def\restr #1{{\restriction_{#1}}}
\def\cyl#1{{\sf c}_{#1}}
\def\diag#1#2{{\sf d}_{#1#2}}

\def\Ra{{\mathfrak{Ra}}}
\def\Ca{{\mathfrak{Ca}}}
\def\set#1{\{#1\} }
\def\Ra{{\mathfrak{Ra}}}
\def\Nr{{\mathfrak{Nr}}}
\def\Tm{{\mathfrak{Tm}}}
\def\A{{\mathfrak{A}}}
\def\B{{\mathfrak{B}}}
\def\C{{\mathfrak{C}}}
\def\D{{\mathfrak{D}}}
\def\E{{\mathfrak{E}}}
\def\A{{\mathfrak{A}}}
\def\B{{\mathfrak{B}}}
\def\C{{\mathfrak{C}}}
\def\D{{\mathfrak{D}}}
\def\E{{\mathfrak{E}}}

\def\P{{\mathfrak{P}}}

\def\Bb{{\mathfrak{Bb}}}
\def\L{{\mathfrak{L}}}
\def\PEA{{\mathbf{PEA}}}
\def\PA{{\mathbf{PA}}}
\def\Bb{{\mathfrak{Bb}}}
\def\L{{\mathfrak{L}}}
\def\CA{{\sf CA}}
\def\RA{{\sf RA}}
\def\RRA{{\sf RRA}}
\def\RCA{{\sf RCA}}

\def\G{{\bf G}}

\def\At{{\sf At}}

\def\Ra{{\sf Ra}}
\def\R{{\sf R}}
\def\Rd{{\sf Rd}}

\def\Sn{{\sf Sn}}
\def\pa{$\forall$}

\def\ws{winning strategy}

\def\PA{{\sf PA}}
\def\PEA{{\sf PEA}}
 \def\RPEA{{\sf RPEA}}
\def\QEA{{\sf QEA}}
\def\SC{{\sf SC}}
\def\Cof{{\sf Cof}}
\def\pe{$\exists$}

\title{Strongly representable algebras}
\author{Tarek Sayed Ahmed}
\begin{document}
\maketitle

\begin{abstract} We give a simpler proof of a result of Hodkinson in the context of a blow and blur up construction
argueing that the idea at heart is similar to that adopted by Andr\'eka et all \cite{sayed}.
The idea is to blow up a finite structure, replacing each 'colour or atom' by infinitely many, using blurs
to  represent the resulting term algebra, but the blurs are not enough to blur the structure of the finite structure in the complex algebra. 
Then, the latter cannot be representable due to a {\it finite- infinite} contradiction. 
This structure can be a finite clique in a graph or a finite relation algebra or a finite 
cylindric algebra. This theme gives example of weakly representable atom structures tthat are not strongly
representable. This is the essence too of construction of Monk like-algebras, one constructs graphs with finite colouring (finitely many blurs),
converging to one with infinitely many, so that the original algebra is also blurred at the complex algebra level, 
and the term algebra is completey representable, yielding a representation of its completion the 
complex algebra. 

A reverse of this process exists in the literature, it builds algebras with infinite blurs converging to one with finite blurs. This idea due to 
Hirsch and Hodkinson, uses probabilistic methods of Erdos to construct a  sequence of graphs with infinite  chromatic 
number one that is $2$ colourable. This construction, which works for both relation and cylindric algebras,
further shows that the class of strongly representable atom structures
is not elementary. We will generalize such a result for any class of algebras between diagonal free algebras 
and polyadic algebras with and without equality,
then we further discuss possibilities for the infinite dimensional case.

Finally, we suggest a very plausible equivalence, and that is:
If $n>2$, is finite, and $\A\in \CA_n$ is countable and atomic, then $\Cm\At\A$ is representable if and only if $\A\in \Nr_n\CA_{\omega}$.
We could prove one side.
\end{abstract}

\section{Variations on a theme}

Unless otherwise specified $n$ will be a finite ordinal $>1$. 
\begin{definition}
\begin{enumarab} 
\item Let $1\leq k\leq \omega$. Call an atom structure $\alpha$ {\it weakly $k$ neat representable}, 
if the term algebra is in $\RCA_n\cap \Nr_n\CA_{n+k}$, but the complex algebra is not representable.
\item Call an atom structure {\it neat} if there is an atomic algebra $\A$, such that $\At\A=\alpha$, $\A\in \Nr_n\CA_{\omega}$ and for every algebra
$\B$ based on this atom structure there exists $k\in \omega$, $k\geq 1$, such that $\B\in \Nr_n\CA_{n+k}$.
\item Let $k\leq \omega$. Call an atom structure $\alpha$ {\it $k$ complete}, 
if there exists $\A$ such that $\At\A=\alpha$ and $\A\in S_c\Nr_n\CA_{n+k}$.
\item Let $k\leq \omega$. Call an atom structure $\alpha$ {\it $k$ neat} if there exists $\A$ such that $\At\A=\alpha$, 
and $\A\in \Nr_n\CA_{n+k}.$ 
\end{enumarab}
\end{definition}
\begin{definition} Let $\K\subseteq \CA_n$, and $\L$ be an extension of first order logic. 
$\K$ is {\it detectable} in $\L$, if for any $\A\in \K$, $\A$ atomic, and for any atom structure 
$\beta$ such that $\At\A\equiv_{\L}\beta$,
if $\B$ is an atomic algebra such that $\At\B=\beta$, then $\B\in \K.$
\end{definition}
We investigate the existence of such structures, and the interconnections. Note that if $\L_1$ is weaker than $\L_2$ and $\K$ is not detectable
in $\L_2$, then it is not detectable in $\L_1$.
We also present several $\K$s and $\L$s as in the second definition. 
All our results extend to Pinter's algebras and quasi polyadic algebras with and without equality.
But first another definition. 

\begin{theorem} 
An atomic algebra $\A\in \RCA_n$ is strongly representable if $\Cm\At\A\in \RCA_n$.
\end{theorem}
This class will be denoted by ${\sf SRCA_n}$; we know that for $n\geq 3$, ${\sf SRCA_n}\subset \RCA_n$. 
The class of completely representable $\CA_n$s, will be denoted by ${\sf CRCA}_n$.

We prove:
\begin{theorem} 
\begin{enumarab}
\item Let $n$ be finite $n\geq 3$. Then there exists a countable weakly $k$ neat atom structure of dimension $n$ if and only if $k<\omega.$
\item There exists an atom structure of a representable atom algebra that is not neat, this works for all dimensions $>1$.
\item For all finite  $n\geq 3,$ there exists an atom structure that is not $n+2$ complete, that is elementary equivalent to an atom structure
that is $m$ neat, for all $m\in \omega$.
\item The class of completely representable algebras, and strongly representable 
ones of dimension $>2$, is not detectable in $L_{\omega,\omega}$ while the class
$\Nr_n\CA_m$ for any ordinals $1<n<m<\omega$, is not detectable 
even in $L_{\infty,\omega}.$ For for infinite $n$, $\Nr_n\CA_m$
is not detectable in first order logic nor in the quantifier free reduct of $L_{\infty,\omega}$.
\end{enumarab}
\end{theorem}
\begin{proof} 
\begin{enumarab}
\item We prove (3) and the first part of (4) together. 
We will not give the details, because the construction we use is a rainbow construction for cylindric algebras,
is really involved and it will be submitted elsewhere,
however, we give the general idea.
We use essentially the techniques in \cite{r}, together with those in \cite{hh}, extending the rainbow construction
to cylindric algebra. But we mention a very important difference. 

In \cite{hh} one game is used to test complete representability.
In  \cite{r} {\it three} games were devised testing different $\Ra$ embeddability properties.
(An equivalence between complete representability and special neat embeddings is proved in \cite{sayed}, se also below)

Here we use only two games adapted to the $\CA$ case. This suffices for our purposes. 
The main result in \cite{hh}, namely, that the class of completely representable algebras of dimension
$n\geq 3$, is non elementary, follows from the fact that \pe\  cannot win the infinite length 
game, but he can win the finite ones. 

Indeed a very useful way of characterizing non-elementary classes, 
say $\K$, is a {\it Koning lemma} argument. The idea is  to devise a game $G(\B)$ on the atom structure of $\B$ such that for a given algebra atomic 
$\A$  \pe\  has a winning strategy on its atom structure for all games of finite length, 
but \pa\ wins the $\omega$ round game. It will follow that there a (countable) cylindric algebra $\A'$ such that $\A'\equiv\A$, 
and \pe\ has a \ws\ in $G(\A')$.
So $\A'\in K$.  But $\A\not\in K$
and $\A\preceq\A'$. Thus $K$ is not elementary.

To obtain our results we use {\it two} distinct games, 
both having $\omega$ rounds, played on a rainbow atom structure, 
the desired algebra is any algebra based on this atom structure; it can be the term algebra generated by the atoms or the 
full complex algebra, or any subalgebra in between.  
Of course the games are very much related.

In this new context \pe\ can also win a finite game with $k$ rounds for every $k$. Here the game
used  is more complicated than that used in Hirsch and Hodkinson \cite{hh}, 
because in the former case we have three kinds of moves which makes it harder for \pe\ 
to win. 

Another difference is that the second game, call it $H$,  is actually  played on pairs, 
the first component is an atomic network (or coloured graph)  defined in the new context of cylindric 
algebras, the second is a set of hyperlabels, the finite sequences of nodes are labelled, 
some special ones are called short, and {\it neat} hypernetworks or hypergraphs are those that label short hyperedges with the same label. 
And indeed a \ws\ for \pe\ 
in the infinite games played on an atom structure $\alpha$ forces that this is the atom structure of a neat reduct; 
that is there is an aatomic algebra $\A\in \Nr_n\CA_{\omega}$, such that $\alpha\cong \At\A$. 
However, unlike complete representability,
does not exclude the fact, in principal, there are other representable algebras 
having the same atom stucture can be only subneat reducts (see item 3 below).

But \pe\ cannot win the infinite length game, it can only win the finite length games of length $k$ for every finite $k$.

On the other hand, \pa\  can win 
{\it another  pebble game}, also in $\omega$ rounds (like in \cite{hh} on a red clique), but
there is a  finiteness condition involved in the latter, namely  is the number of nodes 'pebbles 'used, which is $k\geq n+2$, 
and  \pa\ s \ws\ excludes the neat embeddablity of the algebra in $k$ extra dimensions. This game will be denoted by $F^k$. 

This implies that $\A$ is elementary equivalent to a full neat reduct but it is not in 
$S_c\Nr_n\CA_{n+2}$. This proves the desired. But we should mention that the construction is complicated, because that atom structures
are not like the $\RA$ case, based on simple atoms, namely the colours, but rather on coloured graphs.

The relation algebra used by Hirsch \cite{r}, 
is very similar to the rainbow relation algebra $\A_{\mathbb{N},\mathbb{Z}}$, based on the two graphs 
$\mathbb{N}, \mathbb{Z}$ in the sense of
\cite{HHbook2}
and that used by us is $\M(\mathbb{N})$, in the sense of \cite{HHbook}, and what coresponds to $\mathbb{Z}$ in our context, 
is the indexing set of greens that \pe\ never uses.  \pa\ wins can only win on the red clique $\mathbb{N}.$

A coloured graph is a graph, 
such that every  edge has a rainbow colour, white, black, green, red, but certain $n-1$ tuples, 
or hyperedges, are also coloured by shades of yellow to code the cylindric algebra part.

The forbidden triples are translated to the non-existence of certain 
triangles in the coloured graphs used, so the relation algebra construction is coded in this new rainbow construction.
Games then are played on coloured {\it hypergraphs}, which a graph together with a set of hyperlabels, labelling finite sequences of nodes.
The short hyperedges are constantly labelled, by $\lambda$ say, giving $\lambda$ neat hypegraphs.
The game is played on these.
Now let us concentrate on the graph part of the game $H$ truncated to finite rounds; the crucial move is the cylindrifier move.
\pa\ provides a coloured graph $\Gamma$,  a face $F$, $|F|\leq n-1$, together with a new node, and \pe\ has to respond, 
by extending the given graph to a coloured one with the new node $x$.

Very roughly, if $x\in \Gamma,$ and $k$ is the new node, are appexes of the same cone with base $F$,  
(a cone is a special coloured graph whose sides are coloured by 
greens), then she uses red. 
She will not run out of them, because the game is finite, so \pa\ cannot win this game, if she adopts her \ws\ in $F^m$ 
by forcing a red clique (this is the {\it only} way he can win).
If they are not, then, in labelling edges the usual rainbow strategy is adopted, she tries white, back and then red.
Never use green. Labelling $n-1$ hyperedges by shades of yellow is not very hard.
This works.  
There are two other amalgamation moves, the for one of them \pe\ has no choice, it is completely determined by \pa\ s move.
The third more consists of amalgamating two coloured graphs $\Gamma_1$ and $\Gamma_2$, if $i\in \Gamma_1\sim \Gamma_2$
and $j\in \Gamma_2\sim \Gamma_1$, then the edge $(i,j)$ is coloured like the cylindrfier move; again the usual ranbow strategy is used.
$n-1$ hyperdesges can also be coloured in the amalgamated graph by the appropriate  shades of yellow.
And in fact the Hirsch Hodkinson's main result in \cite{r}, 
can be seen as a special case, of our construction. The game $F^k$, without the restriction on number
of pebbles used and possibly reused, namely $k$ (they have to be reused when $k$ is finite), but relaxing the condition of finitness,
\pa\ does not have to resuse node, and then this game  is identical to the game $H$ when we delete the  hyperlabels from the latter, 
and forget about the second and third 
kinds of move. So to test only complete representability, we use only these latter games, which become one, namely the one used 
by Hirsch and Hodkinson in \cite{hh}. 

In particular, our algebra $\A$ constructed is not completely representable, but is elementary equivalent to one that is.
This also implies that the class of completely representable atom structures are not elementary, the atom structure of the 
former two structures are elementary equivalent, one is completely representable, the other is not.  Hence ${\sf CRCA}_n$ is not detectable
in $L_{\omega,\omega}$.
Since an atom structure of an algebra is first order interpretable in the algebra, hence, 
the latter also gives an example of an atom structure that is 
weakly representable but not strongly representable.

\item Now we prove 2, let $k$ be a cardinal. We give a concrete example of $\bold M$ above.
Let $\E_k=\E_k(2,3)$ denote the relation algebra
which has $k$ non-identity atoms, in which $a_i\leq a_j;a_l$ if $|\{i,j,l\}|\in \{2,3\}$
for all non-identity atoms $a_i, a_j, a_k$.(This means that all triangles are allowed except the monochromatic ones.) 
These algebras were defined by Maddux.
Let $k$ be finite, let $I$ be the set of non-identity atoms of $\E_k(2,3)$ and let $P_0, P_1\ldots P_{k-1}$ be an enumeration of the elements of $I$.
Let $l\in \omega$, $l\geq 2$ and let $J_l$ denote the set of all subsets 
of $I$ of cardinality $l$. Define the symmetric ternary relation on $\omega$ by $E(i,j,k)$ if and only if $i,j,k$ are evenly distributed, that is
$$(\exists p,q,r)\{p,q,r\}=\{i,j,k\}, r-q=q-p.$$
Now assume that $n>2$, $l\geq 2n-1$, $k\geq (2n-1)l$, $k\in \omega$. Let $\bold M=\E_k(2,3).$
Then $\bold M$ is a simple, symmetric finite atomic relation algebra.
Then as above, $\bold M$ is blown up and blurred. This is done by splitting each atom into infinitely countable many ones 
and using a finite set of blurs. So recall by the discussion above, that 
the underlying set of the new atom structure will be of the form
$\omega\times \At\M\times J$, $J$ is a set of finite blurs that corresponds to colours that in turn correpond to non principal 
ultarfilters, needed to represent the term algebra. This term algebra  which is blurred in the sense that $\M$ is not embeddable in it; 
but $\M$ will be embeddable in the full complex algebra the former can be only represented 
on finite sets, the later on infinite sets, if at all, hence
it cannot be represpentable. The idea, as indicated above is to define two partitions of the set $I\times \At\M\times J$, 
the first is used to embed $\M$ into the complex algebra, and the term algebra will be the second partition up to finite and 
cofinite deviations.

Now we get technical again. We have $$(\forall V_2\ldots, V_n, W_2\ldots W_n\in J_l)(\exists T\in J_l)(\forall 2\leq i\leq n)$$
$$(\forall a\in V_i)\forall b\in W_i)(\forall c\in T_i)(a\leq b;c).$$
That is $(J4)_n$ formulated in \cite{ANT} p. 72 is satisfied. Therefore, as proved in \cite{ANT} p. 77,
$B_n$ the set of all $n$ by $n$ basic matrices is a cylindric basis of dimension $n$.
But we also have $$(\forall P_2,\ldots ,P_n,Q_2\ldots Q_n\in I)(\forall W\in J_l)(W\cap P_2;Q_2\cap\ldots \cap P_n:Q_n\neq 0)$$
That is $(J5)_n$ formulated on p. 79 of \cite{ANT} holds. According to definition 3.1 (ii) $(J,E)$ is an $n$ blur for $\M$, 
and clearly $E$ is definable in $(\omega,<)$.
Let $\C$ be as defined in lemma 4.3 in \cite{ANT}.   
Then, by lemma 4.3, $\C$ is a subalgebra of $\Cm\B_n$, hence it contains the term algebra $\Tm\B_n$.
Denote $\C$ by $\Bb_n(\M, J, E)$. Then by theorem 4.6 in \cite{ANT} $\C$ is representable, and by theorem 4.4 in \cite{ANT} 
for $m<n$
$\Bb_m(\M,J,E)=\Nr_m\Bb_n(\M,J,E)$. However $\Cm\B_n$ is not representable.
In \cite{ANT} $\R=\Bb(\M,J,E)$ is proved to be generated by a single element. 
Let $\A_k\in \RCA_n\cap \Nr_n\CA_{n+k}$. 

For the other implication, if $k=\omega$, then algebra in $\Nr_n\CA_{\omega}$ will be completely representable. 
If the term algebra is completely representable, then the complex algebra will be 
representable, and we do not want that.

\item Concerning that the class of strongly representable algebras, one uses an ultraproduct of what we call {\it anti-Monk} algebras.
If one increases the number of blurs in the above construction, then one gets a 
a sequence of non representable algebras, namely the complex algebras
based on the atom structure as defined above, with an increasing number of blurs. 
This corresponds to algebras based on graphs with increasing finite chromatic number ; 
the limit will be an algebra based on a graph of infinite chromatic number, hence
will be representable, in fact, completely representable. This for example proves Monk's classical non finite axiomatizability result.
A graph which has a finite colouring is called a bad graph by Hirsch and Hodkinson. A good graph is one which gives representable algebras,
it has infinite chromatic number. So the Monk theme is to construct algebra based on bad graphs that converge to one that is based
on  a good graph.
This theme is reversed  by used by what we call anti Monk algebras, that are based on Erdos graphs.
Every  graph in this sequence has infinite chromatic number and the limit  algebra based on the ultraproduct of these
graphs will be only  two colourable. This shows that the class of strongly atom structures is not elementary, and 
since an atom structure of an algebra is first order interpretable in the algebra, then ${\sf SRCA_n}$ is not detectable 
in first order logic.

\end{enumarab}

\end{proof}

It is known, this is proved by Sayed Ahmed, and later (independently) by Robin Hirsch that for a countable algebra 
$\A\in \CA_n$ where $n\in \omega$, $\A$ is completely representable if and only if $\A$ is atomic and
$\A\in S_c\Nr_n\CA_{\omega}.$ 

(A form of this characterization exists in the infinite dimensional case, due also to Sayed Ahmed.)
A natural question, cannot help but to come to mind, and that is: 
Can we characterize the class of strongly representable algebras via neat embeddings.

At least we have one side:

\begin{theorem} If $\A$ is countable and atomic, and $\A \in {\sf SRCA_n}$, then 
$\A\in S_c\Nr_n\CA_{\omega}$.
\end{theorem}
\begin{proof} Since ${\sf SRCA_n}\subseteq {\sf CRCA_n}.$
\end{proof}
However, there are completely representable algebras that are not strongly representable.
So the converse is not true. A natural candidate for a complete characterization, at least in the countable case,
is the class $\Nr_n\CA_{\omega}$. That is, for a countable atomic $\A\in \CA_n$, $\A \in {\sf SRCA}_n$ if and only if $\A\in \Nr_n\CA_{\omega}.$
For the remaining direction, assume that $\A\in \Nr_n\CA_{\omega}$ is atomic and countable. We need to 
prove that $\Cm\At\A$ is representable.
If true this gives an entirely new proof that the class is not elementary. In any event we know that any
$K$ such that $\Nr_n\CA_{\omega}\subseteq K\subseteq S_c\Nr_n\CA_{\omega}$ is not elementary
and this too can help, in giving a different proof.

However, the essence of showing that the class ${\sf CRCA_n}$ is not elementary depends on a rainbow construction, 
while the proof that ${\sf SCRA_n}$ is not elementary depends
on anti-Monk algebras, makes one wonder. 
(Though these two techniques are unified in \cite{1}, but the unification is only superficial, it misses on the core of the constructions involved, 
and how they compare, 
if at all).
On the other hand, the blow up and blur construction of Andr\'eka and N\'emeti, discussed in detail
in part one,  gives an example of an atomic algebra in $\RCA_n\cap \Nr_n\CA_{n+k}$ that is not strongly representable.
hence $\RCA_n\cap \Nr_n\CA_{n+k}$ is not closed {\it under completions}, even if this completion is sought in the bigger class $\RCA_n$.
This happens  for any finite $k.$ When $k=\omega$, then $\Nr_n\CA_{\omega}\subseteq \RCA_n$ anyway, 
which prompts:

{\bf Is the class $\Nr_n\CA_{\omega}$ closed under completions, with respect to $\RCA_n$?},

And perhaps an easier question:

{\bf If $\A\in \Nr_n\CA_{\omega}$ is countable and atomic, is it true that $\Cm\At\A\in \RCA_n$.?}

An affirmative answer will confirm our claimed characterization.
Summing up, that would be a very {\it neat} theorem:

\begin{theorem} Let $n\geq 3$ be finite. Let $\A\in \CA_n$ be countable and atomic.
Then
\begin{enumarab}
\item $\A\in {\sf CRCA_n}$ iff $\A\in S_c\Nr_n\CA_{\omega}$
\item $\A\in {\sf SRCA_n}$ iff $\A\in \Nr_n\CA_{\omega}$
\end{enumarab}
\end{theorem}
We have four implications in the above theorem, and we know that three are right. What about the fourth?

\section{Generalizations}

\subsection{Omitting types for multi modal logics}

Constructing weakly representable atom structures that is not strongly representable, 
gives a representable algebra, namely, the term algebra, call it $\A$, that is not completey representable, 
for else the complex algebra will be representable, and so the atom structure will be not be strongly representable after all.
If the atom structure is countable, then $\A$ can be used to given an $L_n$ theory $T$ for which 
the omitting types theorem fails. If $\A$ is simple, then $T$ is complete.

Indeed, assume that $\A=\Fm_T$,  that is $\A$ is the Lindenbaum Tarski quotient algebra corresponding
to the $L_n$ theory $T$, and let $\Gamma$ be the following set of formulas$\{\phi: \phi/T  \text {is a co-atom} \}$. 
Then $\Gamma$ is a non principal-type, because $\sum \At\A=1$, so that $\prod\{-x: x\in \At\A\}=0$,  
but $\Gamma$ cannot be omitted in any model of $T$ , for any such model will be the base
of a complete representable of $\A$. This means that there is no $T$ 
witness for $\Gamma$ in $\L_n$, that is there no $n$ variable formula $\phi$ such that 
$T\models \phi\to \Gamma$.  If $T\models \phi\to \Gamma$, then of course every model of $\phi$ will realize $\Gamma$, $\phi$ witnesses
$\Gamma.$
The omitting types theorem, or rather the contrapositive thereof,  implies that there the converse also hold, namely, 
if $\Gamma$ is realized in every model of $T$, then it has to have a witness, using $k$ variables, possibly $>n$.
But if for every $k\in \omega$, there exists countable $\A\in \RCA_n\cap \Nr_n\CA{n+k}$, 
that is not completely representable, then each such $\A$ gives a countable $L_n$ theory $T$
and a set $\Gamma$ (constructed from the  the co atoms), such $\Gamma$ 
is realized in every model (any model omitting $\Gamma$ will give a complete
representation), but the witness has to use more than $n+k$ variables.  

Do we have an analogous result for fragments of $L_n$ without equality. The answer is yes.
This theorem holds for Pinter's algebras and polyadic algebras, let $\K$ denote either class.
It suffices to show that there exists $\B\in {\sf RK}_n\cap \Nr_n\K_{n+k}$ that is not completely completely representable. 
But this is not hard to show. Let $\A$ be the algebra provided by the main theorem in \cite{ANT}. 
Then first we can expand $\A$ to a polyadic equality
algebra because it is a subalgebra of the complex algebra based on the atom structure of basic matrices.
This new algebra will also be in $\RPEA_n\cap \Nr_n\PEA_{n+k}$. Its reduct, whether the polyadic or the Pinter, will be as desired.

Indeed consider the $\PA$ case, with the other case completely analagous,
this follows from the fact that $\Nr_n\K_n\subseteq \Rd\Nr_n\PEA_{n+k}=\Nr_n\Rd\PEA_{n+k}\subseteq \Nr_n\PA_{n+k}$,
and that $\A$ is completely representable if and only if its diagonal free reduct is, as will be shown below.

Now what if we only have cylindrifiers, that is dealing with $\Df_n$. Let $\A$ be the cylindric algebra constructed above. Assume that 
there a type $\Gamma$,
that is realized in every representation of $\A$ but has no witness using extra $k$ variables. Let $\B=\Rd_{df}\A$. 
Let $f:\B\to \C$ be a diagonal free representation of $\B$. 
The point is that though $\Gamma$ is realized in every {\it cylindric} represenation of $\A$,
there might be a representation of its diagonal free reduct that omits $\Gamma$, these are more, 
because we do not require preservation of the diagonal 
elements. This case definitely needs further research, and we are tempted to think that it is not easy.

But what we are sure
of is that the ordinary omitting types theorem fails
for $L_n$ without equality (that is for $\Df_n$) for $n\geq 3$. One way, among many other, is to construct a
representable countable atomic algebra $\A\in {\sf RDf}_n$, that is not completely representable.
the diagonal free reduct of the cylindric algebra constructed in \cite{ANT} is such.
Now what about $\Df_2$? We do not know. But if we have only {\it one} replacement then it fails.
This can be easily recovered from the above example when $n=2$. 
For higher dimensions, the result follows from the following example from 
\cite{AGNS}.

\begin{example}

Let $\B$ be an atomless Boolean set algebra with unit $U$, that has the following property:
For any distinct $u,v\in U$, there is $X\in B$ such that $u\in X$ and $v\in {}\sim X$.
For example $\B$ can be taken to be the Stone representation of some atomless Boolean algebra.
The cardinality of our constructed algebra will be the same as $|B|$.
Let $$R=\{X\times Y: X,Y\in \B\}$$
and
$$A=\{\bigcup S: S\subseteq R: |S|<\omega\}.$$
Then indeed we have $|R|=|A|=|B|$. 
Then $\A$ is a subalgebra of $\wp(^2U)$.
$$S=\{X\times \sim X: X\in B\}.$$
Thus $$S_0^1(\sum S)=U\times U$$
and
$$\sum \{S_{0}^1(Z): Z\in S\}=\emptyset.$$
For $n>2$, one takes $R=\{X_1\times\ldots\times X_n: X_i\in \B\}$ and the definition of $\A$ is the same. Then,
in this case, one takes $S$ to be
$X\times \sim X\times U\times\ldots\times U$
such that $X\in B$. The proof survives verbatim.
By taking $\B$ to be countable, then $\A$ can be countable, and so it violates the omitting types theorem.
\end{example}

$L_n$ can be looked upon  as a multi-dimensional modal logic, where cylindrifiers its most prominent citizens can be viewed as diamonds. 
Therefore, it is natural to study also omitting types (in relation to complete representations) for multi dimensional 
modal logics that do not have cylindrfiers, but have other
modalities like substitutions. Now let ${\sf SA}_n$ be the cylindrifier free reducts of polyadic algebras. Then:

\begin{theorem} For any ordinal $\alpha>1$, and any infinite cardinal $\kappa$, 
there is an atomic set algebra in ${\sf SA}_n$ with $|A|=\kappa$,  that is not  completely representable. 
In particular, $\A$ can be countable.
\end{theorem}
\begin{proof} Let $\alpha$ be the given ordinal.  Let $|U|=\mu$ be an infinite set and $|I|=\kappa$ be a cardinal such 
that $Q_n$, $n\in \kappa$,  is a family of $\alpha$-ary relations that form a partition of $V={}^{\alpha}U^{(p)}$, 
for some fixed sequence $p\in {}^{\alpha}U$. 
Let $i\in I$, and let $J=I\sim \{i\}$. Then of course $|I|=|J|$. Assume that $Q_i=D_{01}=\{s\in V: s_0=s_1\},$
and that each $Q_n$ is symmetric; that is for any $i,j\in n$, $S_{ij}Q_n=Q_n$. 
It is straightforward to show that such partitions exist.


Now fix $F$ a non-principal ultrafilter on $J$, that is $F\subseteq \mathcal{P}(J)$. 
For each $X\subseteq J$, define
\[
 R_X =
  \begin{cases}
   \bigcup \{Q_k: k\in X\} & \text { if }X\notin F, \\
   \bigcup \{Q_k: k\in X\cup \{i\}\}      &  \text { if } X\in F
  \end{cases}
\]

Let $$\A=\{R_X: X\subseteq I\sim \{i\}\}.$$
Notice that $|\A|\geq \kappa$. Also $\A$ is an atomic set algebra with unit $R_{J}$, and its atoms are $R_{\{k\}}=Q_k$ for $k\in J$.
(Since $F$ is non-principal, so $\{k\}\notin F$ for every $k$). This can be proved exactly like in \cite{AGNS}. 
The subalgebra generated by the atoms is as required. Indeed in such algebras $s_0^1$ is not completely additive, 
and it can be checked that a complete representation forces all operation to be completely additive.
For $n=2$, one requires that each $Q_n$ has domain and range equal to $U$. 
This is not hard to do. 

\end{proof}

We should also mention that this example shows that Pinters algebra for all dimensions may not be completely representable 
answering an implicit question
of Hodkinson's \cite{AU} top of p. 260.

From this we get, contrary to a current wide spread belief that

\begin{theorem} There are atomic $\PA_2$s that are not completely representable; 
furthermore thay can be countable. However, the class of atomic $PA_2$ is elementary
\end{theorem}
\begin{proof}

We modify the above proof as follows. Take each $Q_n$ so it has both domain and range equal to $U$. 
This is possible; indeed it is easy to find such a 
partition of ${}^2U$.
Let $\At(x)$ be the first order formula expressing that $x$ is an atom. That is $\At(x)$ is the formula
$x\neq 0\land (\forall y)(y\leq x\to y=0\lor y=x)$. For distinct $i,j<n$ let $\psi_{i,j}$ be the formula:
$y\neq 0\to \exists x(\At(x)\land {\sf s}_i^jx\neq 0\land {\sf s}_i^jx\leq y).$ Let $\Sigma$ be the polyadic axioms  together
with the $\psi_{i,j}$s, as defined above $(i,j\in \{0,1\})$.
Then it is not hard to check that 
${\bf Mod}(\Sigma)$ is the class of completely representable $\PA_2$s. 
\end{proof}

When we have the diagonal element $D_{01}$ it also does not work for $2$ dimensions. 
(In particular, it does not settle the cylindric case, a task already implemented by 
Andr\'eka and N\'emeti).

\subsection{Generalizing the constructions of Hirsch and Hodkinson to other cylindric-like algebras}

In definition 3.6.3 \cite{HHbook} a cylindric atom structure is defined from a family $K$ of $L$ structures, closed under forming subalgebra. 
This class is  formulated in a language
$L$ of relation symbols having arity $<n$. Call this atom structure $\rho(K)$. 
The atom structure, denote it by ${\cal F}$  can be turned easily into a polyadic equality atom structure 
by defining the binary  relations correponding to the substitution $s_{i,j}$ by: $R_{ij}=\{[f], [g]: f, g \in {\cal F}: f=g\circ [i,j]\}.$

Two examples are given of such classes. For example the (rainbow) class defined in 3.6.9 \cite{HHbook}, 
as an example of  classes
based on on graph.
Fix a graph $\Gamma$. The rainbow polyadic equality algebra based on this graph is denoted by $R(\Gamma)$ 
is the complex algebra of $\rho(K(\Gamma))$, namely $\Cm\rho(K(\Gamma))$.
It is proved that if $\Gamma$ is a countable graph, then the cylindric algebra $R(\Gamma)$ is completely representable
if and only if $\Gamma$ contains a reflexive node or an infinite clique, 
This proof can be very easily checked to work for polyadic equality algebras, it also works for
diagonal free algebras, because arities of relation symbols are $<n$.

Define $K_k$ and $\Gamma$ as in corollary 3.7.1 in \cite{HHbook}. Then $R(\Gamma)$ is  completely representable.
But $\Gamma$ has arbirary large cliques, hence it is 
elementary equivalent to a countable graph $\Delta$ with an infinite clique. Then $R(\Delta)\equiv R(\Gamma)$,
The latter is completely representable, the {\it diagonal free reduct} of the former is not. (A complete representation of the diagonal fre reduct induces a complete representation of the algebra 
itself a result of Hodkinson \cite{AU}.)
Notice that $\Delta\equiv \Gamma$ as first order structures.

The other class of algebras is the generalization of the one we dealt with in this paper. 
It gives examples of strong representability of atom structures built on graphs,
pending on the chromatic number of th graph.

We are in a position to say:

{\it All such results extend to any class of algebra between diagonal free algebras and  polyadic equality case, by observing two things:} 

(1) The results of Hirsch and Hodkinson extend to the polyadic equality case, substitutions corresponding to transpositions 
on all atom structures considered can be defined in an obvious way.

(2) Constructed algebras are based on models
in signatures whose relation symbols have rank $<n$, which means that the algebras are 
generated by elements whose dimension sets are strictly less than $n$;
this makes the passage from the case with equality to non equality
possible.

\begin{theorem} Let $\D$ be a polyadic equality algebra of dimension $n\geq 3$, that is generated by the set $\{x\in D: \Delta x\neq n\}.$
Then if $\Rd_{df}\D$ is completely representable, then so is $\D$.
\end{theorem}
\begin{proof} We asume that that $\D$ is simple, the general case is not much harder. Let $h: \D\to \wp(V)$ be a complete representation, where $V=\prod_{i<n}U_i$ for sets $U_i$.
We can assume that $U_i=U_j$ for all $i,j<n$, and if $s\in V$, $i,j<n$ and $a_i=a_j$ then $a\in h(d_{ij}$.
Let $U$ be the disjoint union of the $U_i$s. 
Let $t_i:U\to U_i$ be the surjection defined by $t_i(u)=(s_j(u))_i$.
Let $g: \D\to \wp(^nU)$ be defined via 
$$d\mapsto \{s\in {}^nU: (t_0(a_0),\ldots, t_{n-1}(a_{n-1}))\in h(d)\}.$$
Then $g$ is a complete representation of $\D.$ Now suppose $s\in {}^nU$, satisfies $s_i=s_j$ with $a_i\in U_k$, say, where $k<n$. 
Let $\bar{b}=s_k(a_i)=s_k(a_j)\in h(\delta).$ Then $t_i(a_i)=b_i$ and $t_j(a_j)=b_j$, so $(t_i(a_i): i<n)$ 
agrees with $\bar{b}$ on coordinates $i,j$. Since $\bar{b}\in h(\delta)$ and 
$\Delta d_{ij}=\{i,j\}$, then $(t_i(a_i): i<n)\in h(d_{ij}$ and so $s\in g(d_{ij}),$ as required.

Now define $\sim_{ij}=\{(a_i, a_j): \bar{a}\in h(d_{ij})$. Then it easy to 
check that $\sim_{01}=\sim_{i,j}$ is an equivalence relation on $U$. For 
$s,t\in {}^nU$, define $s\sim t$, if $s_i\sim t_i$ for each $i<n$, then $\sim$ is an equivalence relation
on $^nU$. Let
$$E=\{d\in D: h(d)\text { is a union of $\sim$ classes }\}.$$
Then $$\{d\in D: \Delta d\neq n\}\subseteq E.$$
Furthermore, $E$ is the domain of a complete subalgebra of $\C$. Hence $E=C$. Now define $V=U/\sim_{01}$, and
define $g:\C\to \wp(^nV)$ via
$$c\mapsto \{(\bar{a}/\sim_{01}): \bar{a}\in h(c).$$
Then $g$ is a complete representation.

Now we drop the assumption that $\D$ is simple. Suppose that $h:\D\to \prod_{k\in K} Q_k$ is a complete representation. 
Fix $k\in K$, let $\pi_k: Q\to Q_k$ be the canonical projection, and let 
$\D_k=rng(\pi_k\circ h)$. We define diagonal elements in $\D_k$ by $d_{ij}=\pi_k(h^{\C}(d_{ij}))$. This expands $\D_k$ 
to a cylindric-type algebra $\C_k$ that is a homomorphic image of $\C$, and hence
is a cylindric algebra with diagonal free reduct $\D_k$. Then the inclusion map $i_k:\D_k\to Q_k$ is a complete
representation of $D_k$. 
Since
$$\pi_k[h[\{c\in C: \Delta c\neq n\}]\subseteq \{c\in C_k: \Delta c\neq n\}$$ 
and $\pi_k, h$ preserve arbitrary sums, then $C_k$ is completely generated by $\{c\in C_k: \Delta c\neq n\}$. 
Now $c_{(n)}x$ is a discriminator term in $Q_k$, so $D_k$ is simple.
So by the above $\C_k$ has complete represenation $g_k:\C_k\to Q_k'$. Define
$g: \C\to \prod_{k\in K}Q_k'$ via
$$g(c)_k=g_k(\pi_k(h(c))).$$
Then $g$ defines a complete representation.
\end{proof}

\subsection{Infinite dimensional case}

The infinite dimensional analogue of such results is completely unknown, and it is stated as an open problem in \cite{hh}.
However, Monk proved non finite axiomatizability of the representable algebras using a lifting argument. 
Here we do the same thing with anti-Monk algebras,
in the hope of getting a weakly representable $\omega$ dimensional atom structure that is not strongly representable.

Let $\Gamma_r$ be a sequence of Erdos graphs. Let $\A_n=\A(n,\Delta_n)$ be the representable atomic algebra of dimension $n$, based  on 
$\Delta_n$, the disjoint union of $\Gamma_r$, $r>n$. Then we know that $\At\A_n$ is a strongly representable atom structure of dimension $n$.
Let $\A_n^+$ be an $\omega$ dimensonal algebra such that 
$\Rd_n\A_{n}^+=\A_n$; we can assume that for $n<m$, there is an $x\in \A_m$, such that  $\A_{n}\cong \Rd\Rl_x\A_{m}$.
Now let $\A=\prod_{n\in F}\A_n^+$, be any non trivial ultarproduct of the $\A_n^+$s, 
then $\A$ is an atomic  $\RCA_{\omega}$ that has an atom structure that is based on the graph $\Delta=\prod \Delta_n$, 
with chromatic number $2$, hence it is only weakly representable.

Another idea is that maybe the construction {\it can} be lifted to infinite dimensions, using {\it weak} set algebras, 
which are representable, based also a labelled the graph $M$ based on $\G$,
replacing $L^n$ by the logic that is like first order logic, but allows infinitary predicates of arity at most $\omega$, call it $L^{\omega}$,
and $L^n_{\infty,\omega}$ by the extention of $L^{\omega}$ with infinite conjunctions.

So we will also have two {\it relativized} set algebras. But there are several difficulties. 
First the graph or model $M$ has to be $n$ homogeneous for every $n$.
We do not think this is a major problem. 
One have to make sure that certain theorems on partial isomorphisms and back and forth systems lifts to the model theory 
of the new logic (we have to be also aware of the fact that we use a different semantics, assignments will have to be eventually constant).

What to our mind, is a major problem is the following. If we insist on including cliques,
then in the limiting case, we need large enough cliques to ensure the representability
of the former algebra, to get back and forth systems for partial isomorphism with finite domain, but arbitrarily large,
but graphs based on such cliques will not have a finite chromatic number, 
which means that we may end up with a representable completion 
as well; Ramseys theorem depending on the existence of a finite colouring of the graph will not apply.

Another idea is that we use Erdos' graphs somehow, and some duality. We have a graph $G$ that is a limit, that is, an ultraproduct of a sequence of graphs $G_r$
each have a finite colouring $>r$, but $G$ has a finite colouring. For each such graph there is a countable 
model $M_r$ defined as above. The desired model on which the two weak set algebras are based is the 
disjoint union of the weak spaces with base $M_r$.


\section{Discussion of  two other longstanding open problems}
\subsection{Completions of subneat reducts}

The existence of a weakly representable atom structure that is not strongly representable of dimension $n\geq 3$,
gives a representable algebra, namely, the term algebra, such that its 
completion the complex algebra is not representable, from which we infer
that the class $\RCA_n$ is not closed under completions, because the latter is the completion of the former. 
We give a sufficient condition that gives  the stronger result for various varieties consisting of subneat reducts, namely $S\Nr_n\CA_k$,
$n\geq 3$, $k>n$.

Whether such varietes are closed under completions, appears an open question in \cite{neat}, and  much earlier in \cite{hh}.
Let $\RA_n$ be the class of subalgebras of atomic relation algebras having $n$ dimensional relational basis.
Then ${\bf S}\Ra\CA_n\subseteq \RA_n$ \cite{HHbook}. 
The full complex algebra of an atom structure $S$
will be denoted by $\Cm S$, and the term algebra by $\Tm S.$
$S$ could be a relation atom structure or a cylindric atom structure.

\begin{theorem}\label{t} Let $n\geq 3$. Assume that for any simple atomic relation algebra $\cal A$ with atom structure $S$, 
there is a cylindric atom structure $H$ such that:
\begin{enumarab}
\item If $\Tm S\in \RRA$, then $\Tm H\in \RCA_n$.
\item $\Cm S$ is embeddable in $\Ra$ reduct of $\Cm H$.
\end{enumarab}
Then for all $k\geq 3$, $S\Nr_n\CA_{n+k}$ is not closed under completions.

\end{theorem}
\begin{demo}{Proof} Let $S$ be a relation atom structure such that $\Tm S$ is representable while $\Cm S\notin \RA_6$.
Such an atom structure exists \cite{HHbook} Lemmas 17.34-17.36 and are finite. 
It follows that $\Cm S\notin {\bf S}\Ra\CA_n$.
Let $H$ be the $\CA_n$ atom structure provided by the hypothesis of the previous theorem.  
Then $\Tm H\in \RCA_n$. We claim that $\Cm H\notin {\bf S}\Nr_n\CA_{n+k}$, $k\geq 3$. 
For assume not, i.e. assume that $\Cm H\in {\bf S}\Nr_n\CA_{n+k}$, $k\geq 3$.
We have $\Cm S$ is embeddable in $\Ra\Cm H.$  But then the latter is in ${\bf S}\Ra\CA_6$
and so is $\Cm S$, which is not the case.
 \end{demo}


\begin{corollary} Assume the hypothesis in \ref{t}. Then the following hold:

\begin{enumarab}

\item There exist two atomic 
cylindric algebras of dimension $n$  with the same atom structure, only one of which is 
representable.

\item For $n\geq 3$ and $k\geq 3$, ${\bf S}\Nr_n\CA_{n+k}$
is not closed under completions and is not atom-canonical. 
In particular, $\RCA_n$ is not atom-canonical.

\item There exists a non-representable $\CA_n$ with a dense representable
subalgebra.

\item For $n\geq 3$ and $k\geq 3$,  ${\bf S}\Nr_n\CA_{n+k}$ 
is not Sahlqvist axiomatizable. In particular, $\RCA_n$ is not Sahlqvist axiomatizable.

\item There exists an atomic representable 
$\CA_n$ with no complete representation.

\end{enumarab}
\end{corollary}
\begin{demo}{Proof} \cite{t}
\end{demo}

Monk and Maddux constructs such an $H$ for $n=3$ and Hodkinson constructs  an $H,$ but $H$ does not satisfy 2 \cite{AU}.

\subsection{Canonical extensions of infinite dimensional representable algebras}

The construction of weakly representable atom structures that are not strongly representable 
also gives an example of an atomic algebra, namely, the term algebra, that is representable but not completely representable.
For if it were completely representable, then this complete representation will induce a representation of the complex algebra.

Monk prove that for finite dimensions $\A$ is representable if its canonical extension $\A^+$ is representable, 
The analogous problem for cylindric algebras of infinite dimensions is a long standing open one, and it is among the open problems given 
in \cite{HHbbok}. We do not even know whether canonical extensions of 
representable algebras are even completely epresentable on {\it weak} units.
In what follows we discuss possibilities of solving this problem, and our discussins will be intervened with non trivial theorems.
that can shed light on the problem in hand. 
We use the notation in \cite{neat}.

\subsubsection{Omitting types again, via Vaught's theorem}

Suppose that $\A\in S_c\Nr_{\alpha}\CA_{\alpha+\omega}$, and $(X_i: i<covK)$ is  a family of subsets of $\A$ 
such that $\prod X_i=0$. Then we can find a weak representation preserving 
those meets using the standard Baire category argument
The statement we prove now is mentioned in \cite{neat}, it is denoted by (*) on p 124  in connection to 
the notion of complete representability and that of neat embeddings,
and again in a slightly different metalogical form, is actually proved  in \cite{Sayed} theorem 3.2.4, in connection to 
the omitting types theorem for certain infinitary extensions of first order
logic; yet again indicating the 
intimate connection. For a Boolean algebra $\A$ and $a\in A$, 
$N_a$ denotes the clopen set in the Stone topology consisting of all ultrafilters
containing $a$.

Assume that $\A\subseteq \Nr_{\alpha}\B$, then because $\A$ is a complete subalgebra, we have $\prod X_i=0$ in $\Nr_{\alpha}\B$, hence also in $\B,$
because the former is the full $\alpha$ neat reduct of $\B$. 
Furthermore, if we take the subalgebra of $\B$ generated by $\A$, then we get
a {\it dimension complemented  algebra}, call it, abusing notation slightly, also $\B$, and in this last algebra the given infinimums are preserved as well.
But by dimension complementedness (which means that the dimension set of any element does not exhaust the dimension), we have
\begin{equation}\label{t1}
\begin{split} (\forall j<\alpha)(\forall x\in A)({\sf c}_jx=\sum_{i\in \alpha\smallsetminus \Delta x}
{\sf s}_i^jx.)
\end{split}
\end{equation}
Here the $s_i^j$ are substitution operations corresponding to replacements; those are definable in cylindric algebras.
Now let $V$ be the weak space $^{\omega}\omega^{(Id)}=\{s\in {}^{\omega}\omega: |\{i\in \omega: s_i\neq i\}|<\omega\}$.
For each $\tau\in V$ for each $i\in \kappa$, let
$$X_{i,\tau}=\{{\sf s}_{\tau}x: x\in X_i\}.$$
Here ${\sf s}_{\tau}$ 
is the unary operation as defined in  \cite[1.11.9]{HMT1}.
For each $\tau\in V,$ ${\sf s}_{\tau}$ is a complete
boolean endomorphism on $\B$ by \cite[1.11.12(iii)]{HMT1}. 
It thus follows that 
\begin{equation}\label{t2}\begin{split}
(\forall\tau\in V)(\forall  i\in \kappa)\prod{}^{\A}X_{i,\tau}=0
\end{split}
\end{equation}
Then from \ref{t1}, \ref{t2}, it follows that for $x\in \A,$ $j<\beta$, $i<\kappa$ and 
$\tau\in V$, the sets 
$$\bold G_{j,x}=N_{{\sf c}_jx}\setminus \bigcup_{i\notin \Delta x} N_{{\sf s}_i^jx}
\text { and } \bold H_{i,\tau}=\bigcap_{x\in X_i} N_{{\sf s}_{\bar{\tau}}x}$$
are closed nowhere dense sets in $S$.
Also each $\bold H_{i,\tau}$ is closed and nowhere 
dense. As before, 
let $$\bold G=\bigcup_{j\in \beta}\bigcup_{x\in B}\bold G_{j,x}
\text { and }\bold H=\bigcup_{i\in \kappa}\bigcup_{\tau\in V}\bold H_{i,\tau.}$$
then $\bold H$ is a countable collection of nowhere dense sets,
and so by the Baire Category theorem,  we get that $H(A)=S\sim \bold H\cup \bold G$ is dense in $S$.
Accordingly let $F$ be an ultrafilter in $N_a\cap X$,
and then define $f$ to the full weak set algebra, call it $\B_a$, with unit $V$ as follows:
$$f(x)=\{ \tau\in {}^{\omega}\omega:  {\sf s}_{\tau}x\in F\}, \text { for } x\in \A.$$ 
Then it can be checked that $f$
is a homomorphism  such that $f(a)\neq 0$, and is as required.
The as before form the product $\prod_{a\in A}\B_a\cong \wp(\bigcup V_a)$, where $V_a$ is the unit of $\B_a$,
$\bigcup$ is their disjoint union,
and finally define the injection $f$ via 
$$x\mapsto (f_a(x): a\neq 0).$$

Let $\L$ be an extension of first order  logic, and $\M$ a model for $\L$, and $T$ an $\L$ theory. Then $\M$ is weak mode for $T$, if
formulas in $\phi$ are satisfied in the Tarskian sense, by assignments in $\M$ that are eventually constant. 

\begin{theorem} The omitting types theorem holds for rich languages if semantics is relativized to weak models, 
furthermore if we have only finitely many relation symbols and 
$T$ is complete, then the number of such non-isomorphic models is exactly $2^{\aleph_0}$.
\end{theorem}
\begin{proof} For the first part, if $If \L$ 
is a countable rich theory and $\Gamma_i$ are non principal types, then $\A=\Fm_T\in {\sf Dc}_{\alpha}$, 
and $X_i=\{\phi_T: \phi\in \Gamma_i\}$ satisfies that $\prod X_i=0.$
But ${\sf \Dc}_{\alpha}\subseteq \Nr_{\alpha}\CA_{\alpha+\omega}$, and we are done.

For the second part, we have $\A$ as above is simple and finitely generated. Let $V={}^{\omega}\alpha^{(Id)}$. Let $\lambda< covK$, 
let $\mathcal{H}(\A)=\bigcap_{i<\omega,x\in A}(N_{-{\sf c}_ix}\cup\bigcup_{j<\omega}N_{{\sf s}^i_jx}),$ 
$\mathcal{H}'(\A)=\mathcal{H}(\A)\cap\bigcap_{i\neq j\in\omega}N_{-{\sf d}_{ij}},$
and $\mathbb{H}= \mathcal{H}'(\A)\cap\sim  \bigcup_{i\in\lambda,\tau\in V}\bigcap_{a\in \Gamma_i}N_{{\sf s}_\tau a)}.$
Here ${\sf s}_{\tau}$ is a well defined boolean endomorphism
of $\A$. (The last union is the models {\it not} omiting the types, it will turn out to be a  set of first category).
This is exactly the space of ultrafilters corresponding to models of $T$ omitting the $\Gamma_i$'s.
Now $\A$ is simple, countable, and is finitely generated, hence it is not atomic, so it has exactly continuum many ultrafilters, and 
$\mathbb{H}=2^{\aleph_0}$, being a $G_{\delta}$ subset of the stone space, which is a Polish space.
This follows from the proprty of $cov K$ that a $\lambda$ union with
$\lambda <covK$ can be reduced to a countable union.
To find the number of non-isomorphic, it remains to  factor out this set of the power of the continuum,
by the relation $F\sim G$ if there exists a finite bijection $\tau\in V$ such that ${\sf s}_{\tau}F=G$
For an ultrafilter $F$, let $h_F(a)=\{\tau \in V: s_{\tau}a\in F\}$, $a\in \A$.
Then $h_F\neq 0$, indeed $Id\in h_F(a)$ for any $a\in F$, hence $h_F$ is an injection, by simplicity of $\A$.
Now $h_F:\A\to \wp(V)$; all the $h_F$'s have the same target algebra.
Then it is easy to check that $h_F(\A)$ is base isomorphic to $h_G(\A)$ iff there exists a finite bijection $\sigma\in V$ such that
$s_{\sigma}F=G$. 
But the orbits of this relation is countable, and and so the number of non-isomorphi models is $|\mathbb{H}/\sim|=2^{\aleph_0}.$
\end{proof}

\subsubsection{A threatening example}

Now why are we giving this proof? If $\A$ is representable, that is $\A\subseteq \Nr_{\alpha}\B$, then $\A^+$ is a {\it complete} subalgebra of 
$\Nr_{\alpha}\B^+$.

However, $\A^+$ is {\it uncountable}, so we cannot apply the last Baire Category argument taking the meet to be preserved 
to be the co atoms of $\A^+$, that is the set $\{-x: x\in \At\A^+\}$, but we thought that the analogy could be inspiring. 
On the other hand, there are {\it omitting types theorems} for {\it uncountable} 
languages when the types are maximal, which is also not the case with the type consisting only 
of co-atoms.

We have managed to construct an uncountable atomic algebra in $\Nr_{\alpha}\CA_{\alpha+\omega}$ 
that is {\it not} completely representable even on {\it weak} units. This could be threatning.
But the big algebra it neatly embeds into is {\it atomless}; which does not apply here.

We give the proof of this, which is not trivial. The proof consists of two parts. 
The first part constructs finite dimensional atomic uncountable algebras
in $\Nr_n\CA_{\omega}$, $n\geq 3$, with no complete representations, then using a lifting argument essentially due
to Monk, implemented via utraproducts,  we lift the result to infinite dimensions. 
The first part is due to Robin Hirsch, the idea is only sketched in a remark \cite{r},  
and remarkably, the example also answers a  question of his in the same paper. 
Therefore we have decided to give a detailed proof, with a slight genearlization; we use any infinite
cardinal $\kappa$ in place of $\omega$.
But first we need.
\begin{theorem}(Erdos-Rado)
If $r\geq 2$ is finite, $k$  an infinite cardinal, then
$$exp_r(k)^+\to (k^+)_k^{r+1}$$

where $exp_0(k)=k$ and inductively $exp_{r+1}(k)=2^{exp_r(k)}$.
\end{theorem}
The above partition symbol describes the following statement. If $f$ is a coloring of the $r+1$
element subsets of a set of cardinality $exp_r(k)^+$
in $k$ many colors, then there is a homogeneous set of cardinality $k^+$+
(a set, all whose $r+1$ element subsets get the same $f$-value).

The proof constructs cylindric algebras from relation algebras having an $\omega$ dimensional cylindric basis.
\begin{theorem} Let $\alpha$ be an infinite ordinal. Then there exists $\B\in \Nr_{\alpha}\CA_{\alpha+\omega}$ 
that is atomic, uncountable and not completely
representable.
\end{theorem}
\begin{proof}
\begin{enumarab}

\item We define an atomic relation algebra $\A$ with uncountably many
atoms. Let $\kappa$ be an infinite cardinal. This algebra will be used to construct cylindric algebras of dimension
$n$ showing that countability is essential in the above characterization.

The atoms are $1', \; a_0^i:i<2^{\kappa}$ and $a_j:1\leq j<
\kappa$, all symmetric.  The forbidden triples of atoms are all
permutations of $(1',x, y)$ for $x \neq y$, \/$(a_j, a_j, a_j)$ for
$1\leq j<\kappa$ and $(a_0^i, a_0^{i'}, a_0^{i^*})$ for $i, i',
i^*<2^{\kappa}.$  In other words, we forbid all the monochromatic
triangles.

Write $a_0$ for $\set{a_0^i:i<2^{\kappa}}$ and $a_+$ for
$\set{a_j:1\leq j<\kappa}$. Call this atom
structure $\alpha$.

Let $\A$ be the term algebra on this atom
structure; the subalgebra of $\Cm\alpha$ generated by the atoms.  $\A$ is a dense subalgebra of the complex algebra
$\Cm\alpha$. We claim that $\A$, as a relation algebra,  has no complete representation.

Indeed, suppose $\A$ has a complete representation $M$.  Let $x, y$ be points in the
representation with $M \models a_1(x, y)$.  For each $i<\omega_1$ there is a
point $z_i \in M$ such that $M \models a_0^i(x, z_i) \wedge a_1(z_i, y)$.

Let $Z = \set{z_i:i<2^{\kappa}}$.  Within $Z$ there can be no edges labeled by
$a_0$ so each edge is labelled by one of the $\kappa$ atoms in
$a_+$.  The Erdos-Rado theorem forces the existence of three points
$z^1, z^2, z^3 \in Z$ such that $M \models a_j(a^1, z^2) \wedge a_j(z^2, z^3)
\wedge a_j(z^3, z_1)$, for some single $j<\kappa$.  This contradicts the
definition of composition in $\A$.

Let $S$ be the set of all atomic $\A$-networks $N$ with nodes
 $\omega$ such that\\ $\set{a_i: 1\leq i<\omega,\; a_i \mbox{ is the label
of an edge in }
 N}$ is finite.
Then it is straightforward to show $S$ is an amalgamation class, that is for all $M, N
\in S$ if $M \equiv_{ij} N$ then there is $L \in S$ with
$M \equiv_i L \equiv_j N.$
Hence the complex cylindric algebra $\Ca(S)\in \CA_\omega$.

Now let $X$ be the set of finite $\A$-networks $N$ with nodes
$\subseteq\omega$ such that
\begin{enumerate}
\item each edge of $N$ is either (a) an atom of
$\c A$ or (b) a cofinite subset of $a_+=\set{a_j:1\leq j<\kappa}$ or (c)
a cofinite subset of $a_0=\set{a_0^i:i<2^{\kappa}}$ and
\item $N$ is `triangle-closed', i.e. for all $l, m, n \in nodes(N)$ we
have $N(l, n) \leq N(l,m);N(m,n)$.  That means if an edge $(l,m)$ is
labeled by $1'$ then $N(l,n)= N(m,n)$ and if $N(l,m), N(m,n) \leq
a_0$ then $N(l,n).a_0 = 0$ and if $N(l,m)=N(m,n) =
a_j$ (some $1\leq j<\omega$) then $N(l,n).a_j = 0$.
\end{enumerate}
For $N\in X$ let $N'\in\Ca(S)$ be defined by
\[\set{L\in S: L(m,n)\leq
N(m,n) \mbox{ for } m,n\in nodes(N)}\]
For $i,\omega$, let $N\restr{-i}$ be the subgraph of $N$ obtained by deleting the node $i$.
Then if $N\in X, \; i<\omega$ then $\cyl i N' =
(N\restr{-i})'$.
The inclusion $\cyl i N' \subseteq (N\restr{-i})'$ is clear.

Conversely, let $L \in (N\restr{-i})'$.  We seek $M \equiv_i L$ with
$M\in N'$.  This will prove that $L \in \cyl i N'$, as required.
Since $L\in S$ the set $X = \set{a_i \notin L}$ is infinite.  Let $X$
be the disjoint union of two infinite sets $Y \cup Y'$, say.  To
define the $\omega$-network $M$ we must define the labels of all edges
involving the node $i$ (other labels are given by $M\equiv_i L$).  We
define these labels by enumerating the edges and labeling them one at
a time.  So let $j \neq i < \omega$.  Suppose $j\in nodes(N)$.  We
must choose $M(i,j) \leq N(i,j)$.  If $N(i,j)$ is an atom then of
course $M(i,j)=N(i,j)$.  Since $N$ is finite, this defines only
finitely many labels of $M$.  If $N(i,j)$ is a cofinite subset of
$a_0$ then we let $M(i,j)$ be an arbitrary atom in $N(i,j)$.  And if
$N(i,j)$ is a cofinite subset of $a_+$ then let $M(i,j)$ be an element
of $N(i,j)\cap Y$ which has not been used as the label of any edge of
$M$ which has already been chosen (possible, since at each stage only
finitely many have been chosen so far).  If $j\notin nodes(N)$ then we
can let $M(i,j)= a_k \in Y$ some $1\leq k < \omega$ such that no edge of $M$
has already been labeled by $a_k$.  It is not hard to check that each
triangle of $M$ is consistent (we have avoided all monochromatic
triangles) and clearly $M\in N'$ and $M\equiv_i L$.  The labeling avoided all
but finitely many elements of $Y'$, so $M\in S$. So
$(N\restr{-i})' \subseteq \cyl i N'$.

Now let $X$ be the set of finite $\A$-networks $N$ with nodes
$\subseteq\omega$ such that
\begin{enumerate}
\item each edge of $N$ is either (a) an atom of
$\c A$ or (b) a cofinite subset of $a_+=\set{a_j:1\leq j<\kappa}$ or (c)
a cofinite subset of $a_0=\set{a_0^i:i<2^{\kappa}}$ and
\item $N$ is `triangle-closed', i.e. for all $l, m, n \in nodes(N)$ we
have $N(l, n) \leq N(l,m);N(m,n)$.  That means if an edge $(l,m)$ is
labeled by $1'$ then $N(l,n)= N(m,n)$ and if $N(l,m), N(m,n) \leq
a_0$ then $N(l,n).a_0 = 0$ and if $N(l,m)=N(m,n) =
a_j$ (some $1\leq j<\omega$) then $N(l,n).a_j = 0$.
\end{enumerate}
For $N\in X$ let $N'\in\Ca(S)$ be defined by
\[\set{L\in S: L(m,n)\leq
N(m,n) \mbox{ for } m,n\in nodes(N)}\]
For $i,\omega$, let $N\restr{-i}$ be the subgraph of $N$ obtained by deleting the node $i$.
Then if $N\in X, \; i<\omega$ then $\cyl i N' =
(N\restr{-i})'$.
The inclusion $\cyl i N' \subseteq (N\restr{-i})'$ is clear.

Conversely, let $L \in (N\restr{-i})'$.  We seek $M \equiv_i L$ with
$M\in N'$.  This will prove that $L \in \cyl i N'$, as required.
Since $L\in S$ the set $X = \set{a_i \notin L}$ is infinite.  Let $X$
be the disjoint union of two infinite sets $Y \cup Y'$, say.  To
define the $\omega$-network $M$ we must define the labels of all edges
involving the node $i$ (other labels are given by $M\equiv_i L$).  We
define these labels by enumerating the edges and labeling them one at
a time.  So let $j \neq i < \omega$.  Suppose $j\in nodes(N)$.  We
must choose $M(i,j) \leq N(i,j)$.  If $N(i,j)$ is an atom then of
course $M(i,j)=N(i,j)$.  Since $N$ is finite, this defines only
finitely many labels of $M$.  If $N(i,j)$ is a cofinite subset of
$a_0$ then we let $M(i,j)$ be an arbitrary atom in $N(i,j)$.  And if
$N(i,j)$ is a cofinite subset of $a_+$ then let $M(i,j)$ be an element
of $N(i,j)\cap Y$ which has not been used as the label of any edge of
$M$ which has already been chosen (possible, since at each stage only
finitely many have been chosen so far).  If $j\notin nodes(N)$ then we
can let $M(i,j)= a_k \in Y$ some $1\leq k < \omega$ such that no edge of $M$
has already been labeled by $a_k$.  It is not hard to check that each
triangle of $M$ is consistent (we have avoided all monochromatic
triangles) and clearly $M\in N'$ and $M\equiv_i L$.  The labeling avoided all
but finitely many elements of $Y'$, so $M\in S$. So
$(N\restr{-i})' \subseteq \cyl i N'$.

Now let $X' = \set{N':N\in X} \subseteq \Ca(S)$.
Then the subalgebra of $\Ca(S)$ generated by $X'$ is obtained from
$X'$ by closing under finite unions.
Clearly all these finite unions are generated by $X'$.  We must show
that the set of finite unions of $X'$ is closed under all cylindric
operations.  Closure under unions is given.  For $N'\in X$ we have
$-N' = \bigcup_{m,n\in nodes(N)}N_{mn}'$ where $N_{mn}$ is a network
with nodes $\set{m,n}$ and labeling $N_{mn}(m,n) = -N(m,n)$. $N_{mn}$
may not belong to $X$ but it is equivalent to a union of at most finitely many
members of $X$.  The diagonal $\diag ij \in\Ca(S)$ is equal to $N'$
where $N$ is a network with nodes $\set{i,j}$ and labeling
$N(i,j)=1'$.  Closure under cylindrification is given.
Let $\c C$ be the subalgebra of $\Ca(S)$ generated by $X'$.
Then $\A = \Ra(\c C)$.
Each element of $\A$ is a union of a finite number of atoms and
possibly a co-finite subset of $a_0$ and possibly a co-finite subset
of $a_+$.  Clearly $\A\subseteq\Ra(\c C)$.  Conversely, each element
$z \in \Ra(\c C)$ is a finite union $\bigcup_{N\in F}N'$, for some
finite subset $F$ of $X$, satisfying $\cyl i z = z$, for $i > 1$. Let $i_0,
\ldots, i_k$ be an enumeration of all the nodes, other than $0$ and
$1$, that occur as nodes of networks in $F$.  Then, $\cyl
{i_0} \ldots
\cyl {i_k}z = \bigcup_{N\in F} \cyl {i_0} \ldots
\cyl {i_k}N' = \bigcup_{N\in F} (N\restr{\set{0,1}})' \in \A$.  So $\Ra(\c C)
\subseteq \A$.
$\A$ is relation algebra reduct of $\c C\in\CA_\omega$ but has no
complete representation.
Let $n>2$. Let $\B=\Nr_n \c C$. Then
$\B\in \Nr_n\CA_{\omega}$, is atomic, but has no complete representation.

\item For $k\geq 3$, let $\c C(k)\in \Nr_k\CA_{\omega}$ be atomic uncountable and not completely representable.
Such algebras exist \cite{Sayed}.
Let $I=\{\Gamma: \Gamma\subseteq \alpha,  |\Gamma|<\omega\}$. 
For each $\Gamma\in I$, let $M_{\Gamma}=\{\Delta\in I: \Gamma\subseteq \Delta\}$, 
and let $F$ be an ultrafilter on $I$ such that $\forall\Gamma\in I,\; M_{\Gamma}\in F$. 
For each $\Gamma\in I$, let $\rho_{\Gamma}$ 
be a one to one function from $|\Gamma|$ onto $\Gamma.$
Let ${\c C}_{\Gamma}$ be an algebra similar to $\CA_{\alpha}$ such that 
$\Rd^{\rho_\Gamma}{\c C}_{\Gamma}={\c C}(|\Gamma|)$. In particular, ${c C}_{\Gamma}$ has an atomic Boolean reduct.
Let  
$\B=\prod_{\Gamma/F\in I}\c C_{\Gamma}$
We will prove that 
$\B\in \Nr_\alpha\CA_{\alpha+\omega},$  $\B$ is atomic and $\B$ is not completely representable. The last two requirements are
easy. $\B$ is atomic, because it is an ultraproduct of atomic algebras. 
$\B$ is not completely representable, even on weak units, because, $\Nr_n\B={\c C}(n)$, and so  such a representaion induces
a complete (square) representation of its $n$ neat reducts, $n\geq 3.$
For the first part, for each $\Gamma\in I$ we know that $\c C(|\Gamma|+k) \in\CA_{|\Gamma|+k}$ and 
$\Nr_{|\Gamma|}\c C(|\Gamma|+k)\cong\c C(|\Gamma|)$.
Let $\sigma_{\Gamma}$ be a one to one function 
 $(|\Gamma|+\omega)\rightarrow(\alpha+\omega)$ such that $\rho_{\Gamma}\subseteq \sigma_{\Gamma}$
and $\sigma_{\Gamma}(|\Gamma|+i)=\alpha+i$ for every $i<\omega$. Let $\A_{\Gamma}$ be an algebra similar to a 
$\CA_{\alpha+\omega}$ such that 
$\Rd^{\sigma_\Gamma}\A_{\Gamma}=\c C(|\Gamma|+k)$.  Then 
$\Pi_{\Gamma/F}\A_{\Gamma}\in \CA_{\alpha+\omega}$.
We prove that $\B= \Nr_\alpha\Pi_{\Gamma/F}\c A_\Gamma$.  Recall that $\B^r=\Pi_{\Gamma/F}\c C^r_\Gamma$ and note 
that $C^r_{\Gamma}\subseteq A_{\Gamma}$ 
(the base of $C^r_\Gamma$ is $\c C(|\Gamma|)$, the base of $A_\Gamma$ is $\c C(|\Gamma|+k)$).
Now here again we use the fact that neat reducts commute with forming ultraproducts, so, for each $\Gamma\in I$,
\begin{align*}
\Rd^{\rho_{\Gamma}}\C_{\Gamma}^r&=\c C((|\Gamma|)\\
&\cong\Nr_{|\Gamma|}\c C(|\Gamma|+k)\\
&=\Nr_{|\Gamma|}\Rd^{\sigma_{\Gamma}}\A_{\Gamma}\\
&=\Rd^{\sigma_\Gamma}\Nr_\Gamma\A_\Gamma\\
&=\Rd^{\rho_\Gamma}\Nr_\Gamma\A_\Gamma
\end{align*}
By the first part of the first part we deduce that 
$\Pi_{\Gamma/F}\C^r_\Gamma\cong\Pi_{\Gamma/F}\Nr_\Gamma\A_\Gamma=\Nr_\alpha\Pi_{\Gamma/F}\A_\Gamma$.
\end{enumarab}
\end{proof}

\subsubsection{Back again to the original problem, coonections with complete representability of polyadic algebras}

One more attempt is to use that principal ultrafilters lie outside nowhere dense sets.
So suppose that  $\A$ is atomic, and the above meets, with just one meet, namely the one consisting of co atoms, and joins hold, 
then the set of principal ultrafilters form a dense 
subset of the Stone space, because $\A$ is atomic, hence we can find  a principal ultrafilter {\it outside the union of such no where dense
sets } and the representation based on it, will be the required complete representation.

We know that $\A^+$ neatly embeds into an algebra in $\omega$ extra dimensions, 
which is the canonical extension of any algebra 
$\B$ in which $\A$ neatly embeds by representability.
So like we did before, we take the subalgebra of the big algebra generated by $\A^+$, we get the required joins and meets by dimension 
complementedness, but we lose atomicity. Unfortunately, dimension complemented algebras are {\it atomless} except for very trivial cases 
which certainly are not ours. 

This idea of  does not survive here,
because it essentially depended on the fact that the principal utrafilters (corresponding to atoms)
are dense. This holds in $\B^+$, but not in its subalgebra generated by $\A^+$ which is atomless.

However, it does work for other important algebras, namely {\it polyadic algebras} of infinite dimension. Any completely additive atomic 
polyadic algebra is completely representable.
Call an algebra $\B\in \PA_{\beta}$ a $\beta$ dilation of $\A\in \PA_{\alpha}$ if $\A=\Nr_{\alpha}\B$.
What plays a major role here, is that if $\A\in \PA_{\alpha}$ then $\A$ has a $\beta$ dilation $\B$ for any $\beta$,
such that if $\A$ is atomic, then so is $\B$. Furthermore, the joins and meets we want to preserve (concerning infinitary cylindrfiers and substitutions)
are preserved in $\B$,  and the principal ultrafilters in $\B$ which are now dense, lie outside nowhere dense sets, 
so that basically the argument used here works. We note that the notions of weakly representable atom structures that are not 
strongly representable  
for $\PA$s is trivial, since every algebra is
representable, and furthermore every completely additive atomic algebra is completely representable. 
However, these notions make sense for polyadic equality algebras of infinite dimensions.

\begin{theorem} If $\A$ is representable, then $\exists$ can win the games $G^{\kappa+\omega}(\A^+)$ 
hence the canonical extension is completely representable.
\end{theorem}
\begin{demo}{Idea} 
Asume that $\A\subseteq \Nr_{\alpha}\B$, then 
$\A^+$ is a complete subalgebra of $\Nr_{\alpha}\B^+.$
\end{demo}

\subsubsection{Omitting types yet, again}

We note that in the above counterexample, the co-atoms do not form an ultrafilter. If they did, then it is impossible to omit them, 
even in the uncountable case. This follows from the next tremendously deep result of Shelah, 
which gets us back to omitting types. This result says that maximal non-principal types can be always 
omitted, as long as they are $< 2^{\aleph_0}.$

\begin{lemma} Suppose that $T$ is a theory,
$|T|=\lambda$, $\lambda$ regular, then there exist models $\M_i: i<{}^{\lambda}2$, each of cardinality $\lambda$,
such that if $i(1)\neq i(2)< \chi$, $\bar{a}_{i(l)}\in M_{i(l)}$, $l=1,2,$, $\tp(\bar{a}_{l(1)})=\tp(\bar{a}_{l(2)})$,
then there are $p_i\subseteq \tp(\bar{a}_{l(i)}),$ $|p_i|<\lambda$ and $p_i\vdash \tp(\bar{a}_ {l(i)})$ ($\tp(\bar{a})$ denotes the complete type realized by 
the tuple $\bar{a}$).
\end{lemma}
\begin{proof} \cite{Shelah} Theorem 5.16, Chapter IV.
\end{proof}
\begin{corollary} For any countable theory, there is a family of $< {}^{\omega}2$ countable models that overlap only on principal types
\end{corollary}
This following provides a proof  result stated in \cite{Sayed} without a proof, namely, theorem
3.2.9. In fact, it proves a stronger result since it addresses a larger class than $\Nr_n\A_{\omega}$ addressed in the above mentioned result,
namely, the class $S_c\Nr_n\CA_{\omega}.$
The inclusion $\Nr_n\CA_{\omega}\subseteq S_c\Nr_n\CA_{\omega}$ is proper; indeed $\B$ constructed above is in the latter class
because it a complete subalgebra of $\A$, but it is not in the former class.
(This is not trivial, for example its relation algebra analogue is an open problem, see \cite{r}. We will investigate this problem below).

\begin{theorem}\label{uncountable} Let $\A=S_c\Nr_n\CA_{\omega}$. Assume that $|A|=\lambda$, where $\lambda$ is an uncountable
cardinal. Let  $\kappa< {}^{\lambda}2$,
and $(F_i: i<\kappa)$ be a system of non principal ultrafilters.
Then there exists
a set algebra $\C$ with base $U$ such that $|U|\leq \lambda$, $f:\A\to \C$ such that $f(a)\neq 0$ and for all $i\in \kappa$, $\bigcap_{x\in X_i} f(x)=0.$

\end{theorem}

\begin{proof}\label{omit} Let $\A\subseteq_c \Nr_n\B$, where $\B$ is $\omega$ dimensional, locally finite and has the same cardinality as $\A$.
This is possible by taking $\B$ to be the subalgebra of which $\A$ is a strong neat reduct generated by  $A$, and noting that we gave countably many
operations.
The $F_i$'s correspond to maximal $n$ types in the theory $T$ corresponding to $\B$, that is, the first order theory $T$ such that $\Fm_T\cong \B$.
Assume that ${\bold F}$ be the given set of non principal ultrafilters, with no model omitting them.
Then for all $i<{}^{\lambda}2$, for every representation $\B_i$ of $\B$,
there exists $F$ such that $F$ is realized in $\B_i$. Let $\psi:{}^{\lambda}2\to \wp(\bold F)$, be defined by
$\psi(i)=\{F: F \text { is realized in  }\B_i\}$.  Then for all $i<{}^{\lambda}2$, $\psi(i)\neq \emptyset$.
Furthermore, for $i\neq j$, $\psi(i)\cap \psi(j)=\emptyset,$ for if $F\in \psi(i)\cap \psi(j)$ then it will be realized in
$\B_i$ and $\B_j$, and so it will be principal. But this means that $\|\bold F|=2^{\aleph_0}$ which is not the case.
So there exists a model omitting the given set of maximal non principla types; algebraically 
there exists $f:\B\to \wp(^{\omega}M)$
such that $\bigcap_{x\in F_i} f(x)=\emptyset.$
The restriction of $f$ to $\A$ defined the obvious way is as required.
\end{proof}

Now we give two metalogical readings of the last two theorems. The first is given in \cite{Sayed}, theorem 3.2.10,
but we include it, because it prepares for the other one, which is an entirely new omitting types theorems for cylindric
algebras of sentences. Cylindrifiers in such algebras can be defined because we include individual constants; the number of
these determines the dimension of the algebra in
question, this interpretation was given in  \cite{Amer}; in the context of representing 
algebras of sentences as (full) neat reducts.

\begin{theorem} Let $T$ be an $\L_n$ consistent theory that admits elimination of quantifiers.
Assume that $|T|=\lambda$ is a regular cardinal.
Let $\kappa<2^{\lambda}$. Let $(\Gamma_i:i\in \kappa)$ be a set of non-principal maximal types in $T$. Then there is a model $\M$ of $T$ that omits
all the $\Gamma_i$'s
\end{theorem}
\begin{demo}{Proof} If $\A=\Fm_T$ denotes the cylindric algebra corresponding to $T$, then since $T$ admits elimination of quantifiers, then
$\A\in \Nr_n\CA_{\omega}$. This follows from the following reasoning. Let $\B=\Fm_{T_{\omega}}$ be the locally finite cylindric algebra
based on $T$ but now allowing $\omega$ many variables. Consider the map $\phi/T\mapsto \phi/T_{\omega}$.
Then this map is from $\A$ into $\Nr_n\B$. But since $T$ admits elimination of quantifiers the map is onto.
The Theorem now follows.
\end{demo}

We now give another natural omitting types theorem for certain uncountable languages.
Let $L$ be an ordinary first order language with
a list $\langle c_k\rangle$ of individual constants
of order type $\alpha$. $L$ has no operation symbols, but as usual, the
list of variables is of order type $\omega$. Denote by $Sn^{L_{\alpha}}$ the set
of all $L$ sentences, the subscript $\alpha$ indicating that we have $\alpha$ many constants Let $\alpha=n\in \omega$.
Let $T\subseteq Sn^{L_0}$ be consistent. Let $\M$  be an $L_0$  model of $T$.
Then any $s:n\to M$ defines an expansion of $\M$ to $L_n$ which we denote by $\M{[s]}$.
For $\phi\in L_n$ let $\phi^{\M}=\{s\in M^n: \M[s]\models \phi\}$. Let
$\Gamma\subseteq Sn^{L_{n}}$.
The question we address is: Is there a model $\M$ of $T$ such that for no expansion $s:n\to M$ we have
$s\in  \bigcap_{\phi\in \Gamma}\phi^{\M}$.
Such an $\M$ omits $\Gamma$. Call $\Gamma$ principal over $T$ if there exists $\psi\in L_n$ consistent with $T$ such that 
$T\models \psi\to \Gamma.$
Other wise $\Gamma$ is non principal over T.

\begin{theorem}  Let $T\subseteq Sn^{L_0}$ be consistent and assume that $\lambda$ is a regular cardinal, and $|T|=\lambda$.
Let $\kappa<2^{\lambda}$. Let $(\Gamma_i:i\in \kappa)$ be a set of non-principal maximal types in $T$.
Then there is a model $\M$ of $T$ that omits
all the $\Gamma_i$'s
That is, there exists a model $\M\models T$ such that there is no $s:n\to \M$
such that   $s\in \bigcap_{\phi\in \Gamma_i}\phi^{\M}$.
\end{theorem}
\begin{demo}{Proof}
Let $T\subseteq Sn^{L_0}$ be consistent. Let $\M$ be an $\L_0$  model of $T$.
For $\phi\in \Sn^L$ and $k<\alpha$
let $\exists_k\phi:=\exists x\phi(c_k|x)$ where $x$ is the first variable
not occurring in $\phi$. Here $\phi(c_k|x)$ is the formula obtained from $\phi$ by
replacing all occurrences of $c_k$, if any, in $\phi$ by $x$.
Let $T$ be as indicated above, i.e, $T$ is a set of sentences in which no constants occur. Define the
equivalence relation $\equiv_{T}$ on $Sn^L$ as follows
$$\phi\equiv_{T}\psi \text { iff } T\models \phi\equiv \psi.$$
Then, as easily checked $\equiv_{T}$ is a
congruence relation on the algebra
$$\Sn=\langle Sn,\land,\lor,\neg,T,F,\exists_k, c_k=c_l\rangle_{k,l<n}$$
We let $\Sn^L/T$ denote the quotient algebra.
In this case, it is easy to see that $\Sn^L/T$ is a $\CA_n$, in fact is an $\RCA_n$.
Let $L$ be as described above. But now we denote it $L_n$,
the subscript $n$ indicating that we have $n$-many
individual constants. Now enrich $L_{n}$
with countably many constants (and nothing else) obtaining
$L_{\omega}$.
Recall that both languages, now, have a list of $\omega$ variables.
For $\kappa\in \{n, \omega\}$
let $\A_{\kappa}=\Sn^{L_{k}}/{T}$.
For $\phi\in Sn^{L_n}$, let $f(\phi/T)=\phi/{T}$.
Then, as easily checked, $f$ is an embedding  of $\A_{n}$
into $\A_{\omega}$. Moreover $f$ has the additional property that
it maps $\A_{n}$, into (and onto) the neat $n$ reduct of $\A_{\beta}$,
(i.e. the set of $\alpha$ dimensional elements of $A_{\beta}$).
In short, $\A_{n}\cong \Nr_{n}\A_{\omega}$. Now again putting $X_i=\{\phi/T: \phi\in \Gamma_i\}$ and using that the
$\Gamma_i$'s are maximal non isolated, it follows that the
$X_i's$ are non-principal ultrafilters
Since $\Nr_n\CA_{\omega}\subseteq S_c\Nr_n\CA_{\omega}$, then our result follows.\end{demo}



\section{On a problem of Robin Hirsch}

For cylindric algebras, we take the $n$ neat reducts of algebras in higher dimension, ending up with a $\CA_n$, 
but we can also take {\it relation algebra reducts}, getting instead a relation algebra. 
The class of relation algebra reducts of cylindric algebras of dimension $n\geq 3$, denoted 
by $\Ra\CA_n$. The $\Ra$ reduct of a $\CA_n$, $\A$, is obtained by taking the $2$ neat reduct of $\A$, then defining composition and converse 
using one space dimension.
For $n\geq 4$, $\Ra\CA_n\subseteq \RA$. Robin Hirsch dealt primarily with this class in \cite{r}.   
This class has also been investigated by many authors, 
like Monk, Maddux, N\'emeti and Simon (A chapter in Simon's dissertation is devoted to such a class, 
when $n=3$). 
After a list of results and publications, Simon proved $\Ra\CA_3$ is not closed under subalgebras for $n=3$, 
with a persucor by Maddux proving the cases $n\geq 5$,
and Monk proving the case $n=4$.

In \cite{r}, Hirsch deals only the relation algebras proving that the $\Ra$ reducts of $\CA_k$s, $k\geq 5$, 
is not elementary, and he ignored the $\CA$ case, probably
because  of analogous results proved by the author on neat reducts \cite{IGPL}.

But  the results in these two last  papers are not identical (via a replacement of relation algebra via a cylindric algebra and vice versa).
There are differences and similarities that are illuminating 
for both. For example in the $\RA$  case Hirsch proved that the elementary subalgebra that is not an $\Ra$ reduct 
is not a complete subalgebra of the one that is.
In the cylindric algebra case, the elementary subalgebra that is not a neat reduct 
constructed is a complete subalgebra of the neat reduct.

Hirsch \cite{r} also proved that any $\K$, such that $\Ra\CA_{\omega}\subseteq \K\subseteq S_c\Ra\CA_k$, $k\geq 5$ 
is not elementary;  but using a rainbow construction for cylindric algebras, we prove its $\CA$ analogue.
In the same paper \cite{r}. In op.cit Robin asks whether the inclusion $\Ra\CA_n\subseteq S_c\Ra\CA_n$ is proper, the construction 
in \cite{IGPL}, shows that for $n$ neat reducts, 
it is. 

Besides giving a unified  proof of all cylindric like algebras for finite dimensions, 
we show that the inclusion is proper given that a certain $\CA_n$ term exists. 
(This is a usual first order formula using $n$ variables).
And indeed  using the technique in \cite{IGPL} we prove an analogous result for relation algebras,
answering the above  question of Hirsch's in \cite{r}. We show that there is an $\A\in \Ra\CA_{\omega}$ with a  an elmentary subalgebra
$\B\in S_c\Ra\CA_{\omega}$, that is not in $\Ra\CA_k$ when $\leq 5$. 
In particular, $\Ra\CA_k\subset S_c\Ra\CA_5$, for
$k\geq 5$.

\begin{theorem} Let $\K$ be any of cylindric algebra, polyadic algebra, with and without equality, or Pinter's substitution algebra.
We give a unified model theoretic construction, to show the following:
\begin{enumarab}
\item For $n\geq 3$ and $m\geq 3$, $\Nr_n\K_m$ is not elementary, and $S_c\Nr_n\K_{\omega}\nsubseteq \Nr_n\K_m.$
\item Assume that there exists a $k$-witness. For any $k\geq 5$, $\Ra\CA_k$ is not elementary
and $S_c\Ra\CA_{\omega}\nsubseteq \Ra\CA_k$.
\end{enumarab}
\end{theorem}

We construct the desired algebra an $\Ra$ reduct of a cylindric algebra. The idea is to 
use an uncountable cylindric algebra $\A\in \Nr_3\CA_{\omega}$, hence $\A$ is representable, together
with a finite atom structure of another simple cylindric algebra, that is also representable.

The former algebra will be a set algebra based on a homogeneous model, that admits elimination of quantifiers
(hence will be a full neat reduct).

Such a model is constructed using  Fraisse's methods of building models by amalgamating smaller parts.
The idea is the smae idea that we used before. 
The Boolean reduct of $\A$ can be viewed as a finite direct product of the of disjoint Boolean relativizations of $\A$.
Each component will be still uncountable; the product will be indexed by the elements of the atom structure.
The language of Boolean algebras can now be expanded by constants also indexed by the atom structure,
so that $\A$ is first order interpretable in this expanded structure $\P$ based on the finite Boolean product.
The interpretation here is one dimensional and quantifier free.

The $\Ra$ reduct of $\A$ be as desired; it will be a full $\Ra$ reduct of a full neat reduct of an $\omega$ 
dimensional algebra, hence an $\Ra$ reduct
of an $\omega$ dimensional algebra, and it has a
complete elementary equivalent subalgebra not in
$\Ra\CA_k$. (This is the same idea for $\CA$, but in this case, and the other cases of its relatives, one spare dimension suffices.)

This {\it elementary subalgebra} is obtained from $\P$, by replacing one of the components of the product with an elementary
{\it countable} Boolean subalgebra, and then giving it the same interpretation.
First order logic will not see this cardinality twist, but a suitably chosen term
$\tau_k$ not term definable in the language of relation algebras will, witnessing that the twisted algebra is not in $\Ra\CA_k$.

\begin{definition}
Let $k\geq 4$. A $k$ witness $\tau_k$ is $m$-ary term of $\CA_k$ with rank $m\geq 2$ such 
that $\tau_k$ is not definable in the language of relation algebras (so that $k$ has to be $\geq 4$)
and for which there exists a term $\tau$ expressible in the language of relation algebras, such that
$\CA_k\models \tau_k(x_1,\ldots x_m)\leq \tau(x_1,\ldots x_m).$ (This is an implication between two first order formulas using $k$-variables).

Furthermore, whenever $\A\in {\bf Cs}_k$ (a set algebra of dimension $k$) is uncountable,
and $R_1,\ldots R_m\in A$  are such that at least one of them is uncountable,
then $\tau_k^{\A}(R_1\ldots R_m)$ is uncountable as well.
\end{definition}

The following lemma, is available in \cite{Sayed} with a sketch of proof; it is fully 
proved in \cite{MLQ}. If we require that a (representable) algebra be a neat reduct,
then quantifier elimination of the base model guarantees this, as indeed illustrated in our fully proved next lemma.

\begin{lemma} Let $V=(\At, \equiv_i, {\sf d}_{ij})_{i,j<3}$ be a finite cylindric atom structure,
such that $|\At|\geq |{}^33.|$
Let $L$ be a signature consisting of the unary relation
symbols $P_0,P_1,P_2$ and
uncountably many tenary predicate symbols.
For $u\in V$, let $\chi_u$
be the formula $\bigwedge_{u\in V}  P_{u_i}(x_i)$.
Then there exists an $L$-structure $\M$ with the following properties:
\begin{enumarab}

\item $\M$ has quantifier elimination, i.e. every $L$-formula is equivalent
in $\M$ to a boolean combination of atomic formulas.

\item The sets $P_i^{\M}$ for $i<n$ partition $M$, for any permutation $\tau$ on $3,$
$\forall x_0x_1x_2[R(x_0,x_1,x_2)\longleftrightarrow R(x_{\tau(0)},x_{\tau(1)}, x_{\tau(2)}],$

\item $\M \models \forall x_0x_1(R(x_0, x_1, x_2)\longrightarrow
\bigvee_{u\in V}\chi_u)$,
for all $R\in L$,

\item $\M\models  \exists x_0x_1x_2 (\chi_u\land R(x_0,x_1,x_2)\land \neg S(x_0,x_1,x_2))$
for all distinct tenary $R,S\in L$,
and $u\in V.$

\item For $u\in V$, $i<3,$
$\M\models \forall x_0x_1x_2
(\exists x_i\chi_u\longleftrightarrow \bigvee_{v\in V, v\equiv_iu}\chi_v),$

\item For $u\in V$ and any $L$-formula $\phi(x_0,x_1,x_2)$, if
$\M\models \exists x_0x_1x_2(\chi_u\land \phi)$ then
$\M\models
\forall x_0x_1x_2(\exists x_i\chi_u\longleftrightarrow
\exists x_i(\chi_u\land \phi))$ for all $i<3$

\end{enumarab}
\end{lemma}
\begin{proof}\cite{MLQ}
\end{proof}

\begin{lemma}\label{term}
\begin{enumarab}

\item For $\A\in \CA_3$ or $\A\in \SC_3$, there exist
a unary term $\tau_4(x)$ in the language of $\SC_4$ and a unary term $\tau(x)$ in the language of $\CA_3$
such that $\CA_4\models \tau_4(x)\leq \tau(x),$
and for $\A$ as above, and $u\in \At={}^33$,
$\tau^{\A}(\chi_{u})=\chi_{\tau^{\wp(^nn)}(u).}$

\item For $\A\in \PEA_3$ or $\A\in \PA_3$, there exist a binary
term $\tau_4(x,y)$ in the language of $\SC_4$ and another  binary term $\tau(x,y)$ in the language of $\SC_3$
such that $PEA_4\models \tau_4(x,y)\leq \tau(x,y),$
and for $\A$ as above, and $u,v\in \At={}^33$,
$\tau^{\A}(\chi_{u}, \chi_{v})=\chi_{\tau^{\wp(^nn)}(u,v)}.$


\end{enumarab}
\end{lemma}

\begin{proof}

\begin{enumarab}

\item For all reducts of polyadic algebras, these terms are given in \cite{FM}, and \cite{MLQ}.
For cylindric algebras $\tau_4(x)={}_3 {\sf s}(0,1)x$ and $\tau(x)={\sf s}_1^0{\sf c}_1x.{\sf s}_0^1{\sf c}_0x$.
For polyadic algebras, it is a little bit more complicated because the former term above is definable.
In this case we have $\tau(x,y)={\sf c}_1({\sf c}_0x.{\sf s}_1^0{\sf c}_1y).{\sf c}_1x.{\sf c}_0y$, and 
$\tau_4(x,y)={\sf c}_3({\sf s}_3^1{\sf c}_3x.{\sf s}_3^0c_3y)$.

\item  We omit the construction of such terms. But from now on, we assme that they exist.
\end{enumarab}
\end{proof}

\begin{theorem}
\begin{enumarab}
\item There exists $\A\in \Nr_3\QEA_{\omega}$
with an elementary equivalent cylindric  algebra, whose $\SC$ reduct is not in $\Nr_3\SC_4$.
Furthermore, the latter is a complete subalgebra of the former.

\item Assume that there is $k$ witness. Then there exists a 
relation algebra $\A\in \Ra\CA_{\omega}$, with an elementary equivalent relation algebra not in $\Ra\CA_k$.
Furthermore, the latter is a complete subalgebra of the former.
\end{enumarab}
\end{theorem}

\begin{proof} Let $\L$ and $\M$ as above. Let
$\A_{\omega}=\{\phi^M: \phi\in \L\}.$
Clearly $\A_{\omega}$ is a locally finite $\omega$-dimensional ylindric set algebra.
$\P$ denotes the
structure $\prod_{u\in {}V}\A_u$ for the signature of Boolean algebras expanded
by constant symbols $1_{u}$ for $u\in V$
and ${\sf d}_{ij}$ for $i,j\in 3$ as in \cite{Sayedneat}.
For the second part; for relation algebras.
The $\Ra$ reduct of $\A$ is a generalized reduct of $\A$, hence $\P$ (as define above) is first order interpretable in $\Ra\A$, as well.
It follows that there are closed terms and a formula $\eta$ built out of these closed terms such that
$$\P\models \eta(f(a), b, c)\text { iff }b= f(a\circ^{\Ra\A} c),$$
where the composition is taken in $\Ra\A$.
Here $\At$ defined depends on $\tau_k$ and $\tau$, so we will not specify it any further,
we just assume that it is finite.

As before, for each $u\in \At$, choose any countable Boolean elementary
complete subalgebra of $\A_{u}$, $\B_{u}$ say.
Le $u_i: i<m$ be elements in $\At$, and let $v=\tau(u_1,\ldots u_m)$.
Let $$Q=(\prod_{u_i: i<m}\A_{u_i}\times \B_{v}\times \times \B_{\breve{b}}\times \prod_{u\in {}V\smallsetminus \{u_1,\ldots u_m, v, \breve{v}\}}
\A_u), t_j)_{u,v\in {}V,i,j<3}\equiv$$
$$(\prod_{u\in V} \A_u, 1_{u,v}, {\sf d}_{ij})_{u\in V, i,j<3}=\P.$$

Let $\B$ be the result of applying the interpretation given above to $Q$.
Then $\B\equiv \Ra\A$ as relation  algebras, furthermore $\Bl\B$ is a complete subalgebra of $\Bl\A$.
Now we use essentially the same argument. We force the $\tau(u_1,\ldots u_m)$
component together with its permuted versions (because we have converse) countable;
the resulting algebra will be a complete elementary subalgebra of the original one, but $\tau_k$
will force our twisted countable component to be uncountable, arriving at a contradiction.

In more detail, assume for contradiction that $\B=\Ra\D$ with $\D\in \CA_k$.
Then $\tau_k^{\D}(f(\chi_{u_1}),\ldots f(\chi_{u_n}))$, is uncountable in $\D$.
Because $\B$ is a full $\RA$ reduct,
this set is contained in $\B.$  For simplicity assume that $\tau^{\Cm\At}(u_1\ldots u_m)=Id.$
On the other hand, for $x_i\in B$, with $x_i\leq \chi_{u_i}$, let $\bar{x_i}=(0\ldots x_i,\ldots)$ 
with $x_i$ in the $uth$ place.
Then we have
$$\tau_k^{\D}(\bar{x_1},\ldots \bar{x_m})\leq \tau(\bar{x_1}\ldots \bar{x_m})\in \tau(f(\chi_{u_1}),\ldots f({\chi_{u_m}}))
=f(\chi_{\tau(u_1\ldots u_m)})=f(\chi_{Id}).$$
But this is a contradiction, since  $\B_{Id}=\{x\in B: x\leq \chi_{Id}\}$ is  countable and $f$ is a Boolean isomorphism.
\end{proof}

\end{document}